\documentclass[11pt]{amsart}
\usepackage[nohug]{diagrams}
\usepackage{amscd}
\usepackage[mathscr]{eucal}
\usepackage{amsfonts}
\usepackage{amsmath}
\usepackage{amsthm}
\usepackage{amssymb}
\usepackage{latexsym}
\usepackage{tabularx,array}
\newfont{\msam}{msam10}

\newtheorem{theorem}[]{Theorem}
\newtheorem{proposition}[]{Proposition}
\newtheorem{corollary}[]{Corollary}
\newtheorem{lemma}[]{Lemma}
\theoremstyle{definition}
\newtheorem{definition}[]{Definition}
\newtheorem{remark}[]{Remark}

\let\nc\newcommand

\nc{\la}{\label}

\def\bthm{\begin{theorem}}
\def\ethm{\end{theorem}}
\def\blemma{\begin{lemma}}
\def\elemma{\end{lemma}}
\def\bproof{\begin{proof}}
\def\eproof{\end{proof}}
\def\bprop{\begin{proposition}}
\def\eprop{\end{proposition}}
\def\bcor{\begin{corollary}}
\def\ecor{\end{corollary}}

\def\ms#1{\mathcal{#1}}

\def\Z{\mathbb{Z}}
\def\R{\mathbb{R}}
\def\H{\mathbb{H}}

\def\O{\mathcal{O}}

\def\Q{\mathbb{Q}}

\def\N{\mathbb{N}}

\def\D{{\mathscr D}}

\def\RR{\mathcal{R}}

\def\CM{{\sf CM}}

\def\AA{{\mathcal A}}

\def\m{\mathfrak{m}}
\def\c{\mathbb{C}}

\def\CC{\mathcal{C}}
\def\bCC{\bar{\mathcal C}}

\nc{\Hom}{{\rm{Hom}}}
\nc{\Ext}{{\rm{Ext}}}
\nc{\HOM}{\underline{\rm{Hom}}}
\nc{\EXT}{\underline{\rm{Ext}}}
\nc{\TOR}{\underline{\rm{Tor}}}
\nc{\End}{{\rm{End}}}
\nc{\GL}{{\rm{GL}}}
\nc{\PGL}{{\rm{PGL}}}
\nc{\SL}{{\rm{SL}}}
\nc{\PSL}{{\rm{PSL}}}
\nc{\Rep}{{\rm{Rep}}}
\nc{\ad}{{\rm{ad}}}
\nc{\dlim}{\varinjlim}
\newcommand{\Lotimes}{\stackrel{\boldsymbol{L}}{\otimes}}

\newcommand{\Mod}{{\tt{Mod}}}

\newcommand{\Tor}{{\rm{Tor}}}

\newcommand{\Pic}{{\rm{Pic}}}
\newcommand{\Aut}{{\rm{Aut}}}
\newcommand{\Auteq}{{\rm{Auteq}}}

\newcommand{\rk}{{\rm{rk}}}

\newcommand{\Ker}{{\rm{Ker}}}

\newcommand{\im}{{\rm{Im}}}

\newcommand{\ei}{e_{\infty}}

\newcommand{\into}{\,\,\hookrightarrow\,\,}

\newcommand{\onto}{\,\,\twoheadrightarrow\,\,}

\newcommand{\rrdaha}{\varepsilon H \varepsilon}
\newcommand{\lldaha}{\varepsilon' H \varepsilon'}
\newcommand{\rldaha}{\varepsilon H \varepsilon'}
\newcommand{\lrdaha}{\varepsilon' H \varepsilon}

\def\x{x^{\pm 1}}
\def\y{y^{\pm 1}}

\begin{document}

\title{The Picard Group of a Noncommutative Algebraic Torus}
\author{Yuri Berest}
\address{Department of Mathematics,
 Cornell University, Ithaca, NY 14853-4201, USA}
\email{berest@math.cornell.edu}
\author{Ajay Ramadoss}
\address{Departement Mathematik,
Eidgenossische TH Z\"urich,
8092 Z\"urich, Switzerland}
\email{ajay.ramadoss@math.ethz.ch}
\author{Xiang Tang}
\address{Department of Mathematics, Washington University, St.
Louis, MO 63139, USA} \email{xtang@math.wustl.edu}
\begin{abstract}
Let $\, A_q := \c\langle \x, \y\rangle/(xy-qyx)\,$. 
Assuming that $q$ is not a root of unity, we compute the Picard group $\, \Pic(A_q) \,$ of the algebra $A_q$, 
describe its action on the space $ \RR(A_q) $ of isomorphism classes of rank 1 projective modules and 
classify the algebras Morita equivalent to $ A_q $. Our computations are based on a `quantum' version of 
the Calogero-Moser correspondence relating projective $A_q$-modules
to irreducible representations of the double affine Hecke algebras $\,{\mathbb H}_{t, q^{-1/2}}(S_n)\,$ at $ t = 1 $.
We show that, under this correspondence, the action of $\, \Pic(A_q) \,$ on $ \RR(A_q) $ agrees with the action of
$ \SL_2(\Z) $ on $\,{\mathbb H}_{t, q^{-1/2}}(S_n)\,$ constructed by Cherednik \cite{C1, Che}. 
We compare our results with smooth and analytic cases. In particular, when $ |q| \ne 1 $, we find that $\,
\Pic(A_q) \cong \text{Auteq}(\D^b(X))/{\Z} \,$, where $ \D^b(X) $ is the bounded derived category
of coherent sheaves on the elliptic curve $ X = \c^*\!/q^{\Z} $.
\end{abstract}
\maketitle

\section{Introduction}

Let $ \c\langle x^{\pm 1}, y^{\pm 1}\rangle $ be the group algebra of the free group on two generators
$x$ and $y$, with coefficients in $\c$. For a fixed parameter $\,q \in \c \,$, we define 
$\,A_q := \c\langle x^{\pm 1}, y^{\pm 1}\rangle/(xy-qyx)\,$. Unless specified otherwise, we will 
assume in this paper that
\begin{equation}\la{nu}
q^n \ne 1 \quad \mbox{for all}\quad n \in {\mathbb N}\ .
\end{equation}
Under this condition, we will $\,(a)$ classify f.~g. projective modules of $A_q$,
$\,(b)$ compute the Picard group $ \Pic(A_q) $ and 
describe its action on projective modules, and $\,(c)$\ classify the algebras Morita equivalent 
to $ A_q $.

There are several reasons to clarify these questions. 
First, $\,A_q$ may be thought of as a ring 
of functions on a noncommutative {\it algebraic} torus. Now, for a noncommutative {\it smooth} torus $ \AA_q $, 
the answers to $(a)$, $(b)$ and $ (c) $ are well known and well understood (mostly thanks to the work of Rieffel, 
see \cite{R1, R2, R3, K}). Geometrically, the algebras $ \AA_q $ arise as deformations of the ring 
$\,{\mathcal C}^\infty({\mathbb T})\,$ of smooth functions on the two-dimensional torus 
$\,{\mathbb T} = {\mathbb S}^1 \times {\mathbb S}^1\,$, and as such, these are fundamental examples of  
noncommutative differentiable manifolds in the sense of A.~Connes \cite{C}.
On the other hand, algebraically, $ \AA_q $ is just a certain completion of $ A_q $, and it is natural to 
ask how the projective modules, Picard groups, Morita equivalences, etc. behave under this completion.
One might expect that $ A_q $ is `too rigid' compared to $ \AA_q $, and the answers to the above questions
are trivial. We will demonstrate that this is not the case, and although the properties of 
$ A_q $ and $ \AA_q $ are indeed very different, the answers to $(a)$, $(b)$ and $ (c) $ in the algebraic 
case are at least as meaningful and interesting as in the smooth case.

Second, $\, A_q $ may be viewed as a `quantum' (or multiplicative) Weyl algebra. Under the assumption 
\eqref{nu}, the ring-theoretic properties of $ A_q $ are indeed similar to those of the Weyl algebra
$ A_1 = \c\langle x, y\rangle/(xy-yx-1) $ (see \cite{J}). From this perspective, our answer to $ (a) $ should 
not be very surprising and, in fact, is not really new. It is known that the rank one projective modules of $ A_1 $ 
are isomorphic to ideals and as such, can be parametrized by certain smooth algebraic varieties $ \CC_n $ called the 
{\it Calogero-Moser spaces} (see \cite{BW1, BW2, BC}). Similarly, the ideals of $A_q$ are described by a certain 
`quantum' version of the Calogero-Moser spaces $ \CC_n^q $  (see Theorem~\ref{MTh}), and in fact, a similar geometric description 
exists for more general classes of noncommutative algebras (see, e.g., \cite{KKO, NvdB, NS}). What is surprising is the fact 
that, unlike in the Weyl algebra case, the classification of ideals of $A_q$ allows one to compute the Picard group $ \Pic(A_q) $.
By comparison, the Picard group $ \Pic(A_1) $ of the Weyl algebra was first computed by Stafford,
using reduction to fields of characteristic $ p > 0 $ (see \cite{Sta}). A different proof can be found 
in \cite{CH}: it relies on Dixmier's classification \cite{D} of maximal commutative subalgebras of $ A_1 $. 
Both proofs are fairly involved and indirect; in particular, they do not follow from the results of \cite{BW1, BW2, BC}. 
To the best of our knowledge, the group $ \Pic(A_q) $ has not appeared in the literature: we therefore consider 
its computation (Theorem~\ref{Pic}) together with a  related Morita classification (Theorem~\ref{Mor}) as 
the main results of this paper.

Third, the quantum Calogero-Moser spaces $ \CC_n^q $ parametrizing the ideals of $ A_q $ can be defined as 
the spectra of spherical subalgebras of Cherednik's double affine Hecke algebra (DAHA) $\,\H_{1,\, q^{-1/2}}(S_n)\,$ (see \cite{O}).
Our results then imply that there is a natural bijection between the set of equivalence classes of 
irreducible representations of $\,{\mathbb H}_{1, q^{-1/2}}(S_n)\,$ amalgamated for all $\,n\,$ and the set 
$ \RR(A_q) $ of isomorphism classes of ideals of $ A_q $. In the rational
case, the analogous bijection was first observed in \cite{EG}, and its conceptual explanation was given in \cite{BCE}. 
In this paper, we will extend the construction of \cite{BCE} to the quantum case (see Section~\ref{S5}). 
A new interesting observation is that, under the Calogero-Moser correspondence, the action of $\, \Pic(A_q) \,$ on 
$ \RR(A_q) $ corresponds to the action of $ \SL_2(\Z) $ on $\,{\mathbb H}_{1,\, q^{-1/2}}(S_n)\,$ 
constructed (and exploited in many applications) by Cherednik \cite{C1, Che}.

Finally, $ A_q $ are fundamental examples of noncommutative algebras which play a role in many areas of mathematics 
and physics. As a specific motivation to clarify questions $(a)$, $(b)$ and $(c)$  for $ A_q $, we mention a recent 
appearance of this algebra in knot theory: it is shown in  \cite{FGL} (see also \cite{G}) that the classical invariants 
of knots -- the so-called $A$-polynomials -- can be naturally quantized, and the corresponding quantizations 
are given by certain {\it noncyclic} ideals of $A_q$. 
Comparing these ideals for different knots is not an easy problem. Having a general classification and canonical 
forms for all the ideals of $A_q$ is certainly helpful in this context. For further discussion and application of our results
we refer to \cite{S}. 

As another application, we should mention the link to integrable systems. The quantum Calogero-Moser spaces $ \CC_n^q $ 
first appeared in \cite{FR} in connection with 2D Toda hierarchy and the Ruijsenaars-Schneider system; since then they
were discussed in numerous papers on integrable systems. By analogy with the Weyl algebra (see \cite{W, BW1}), the ideals 
of $A_q $ are related to algebraic solutions of these systems, which in turn can be described in terms of 
$q$-version of the adelic Grassmannian of G.Wilson \cite{W, W2} (see \cite{CH1}, \cite{CN} and \cite{BC1} for more details). 

The paper is organized as follows. In Section~\ref{S2}, we define the Calogero-Moser spaces $ \CC^q_n $
and give our classification of ideals of $ A_q $. In Section~\ref{S3}, we compute
the Picard group $ \Pic(A_q) $ and discuss some implications. In Section~\ref{S4}, we describe the action of 
$ \Pic(A_q) $ on the Calogero-Moser spaces and classify the algebras Morita equivalent to $ A_q $ in terms 
of this action. In Section~\ref{S5}, we explain the relation between the DAHA and $ A_q $ and outline proofs 
of our main results. Finally, in Section~\ref{S6}, we compare the properties of $A_q$ with the properties 
of noncommutative smooth tori $ \AA_q $.

\subsection*{Acknowledgements}
We are very grateful to O.~Chalykh for his remarks and suggestions; although he is formally not a coauthor,
many ideas in this paper belong to him (in particular, Theorem~\ref{MTh} in its current form first appeared 
in \cite{BC1}, where it was proven by elementary methods). We are also grateful to M.~Rieffel for his encouragement 
to write this paper and for providing us with references \cite{K}, \cite{K2} and \cite{Wa}. 
The first author would like to thank J.~Alev, A.~Polishchuk and G.~Wilson for interesting discussions and comments.

The work of Yu.~B. and X.~T. was partially supported by NSF grants DMS 0901570 and DMS 0900985, 
respectively. A.~R. is currently funded by the Swiss National Science Foundation (Ambizione
Beitrag Nr. PZ00P2-127427/1).

\section{The Calogero-Moser Correspondence for $A_q$}
\la{S2}

Under the assumption \eqref{nu}, the algebra $ A_q$ is a simple Noetherian
domain of homological dimension $1$ (see \cite{J}, Theorem~2.1). In that case, it is
known that every right projective module of rank $ \ge 2 $ is free,
while every rank one projective module is isomorphic to a right ideal of $ A_q$
(see \cite{We}, Theorem~1). Thus the problem of classifying projectives over $A_q$
reduces to classifying the right ideals of $ A_q $. Below, we will give an explicit
classification of such ideals similar to the classification of ideals in the Weyl algebra 
$A_1 $.

For an integer $\, n \geq 1 \,$, let $ \tilde{\CC}^q_n $ denote the
space of matrices
$$
\{\,(X,\, Y,\, i,\, j) \,:\, X,\, Y \in \GL_n(\c) \, , \, i \in
\c^n, \, j \in \Hom(\c^n,\, \c)\,\} \ ,
$$
satisfying the equation
\begin{equation}
\la{qrone}
qXY-YX+ij=0\ .
\end{equation}
The group $\, \GL_n(\c) \,$ acts on $\, \tilde{\CC}^q_n \,$ in the
natural way:
\begin{equation}
\la{gln}
(X, Y, i, j) \mapsto (g X g^{-1}, \,g Y g^{-1},\, g i,\, j g^{-1}) ,
\quad g \in \GL_n(\c)\ ,
\end{equation}
and this action is free for all $ n $. We define the {\it $n$-th Calogero-Moser space}
$\, \CC^q_n \,$ to be the quotient variety
$\, \tilde{\CC}^q_n/\GL_n(\c) \,$ modulo \eqref{gln}.
$\, \CC^q_n \,$ are smooth irreducible affine symplectic varieties
of (complex) dimension $ 2n $ (see \cite{O}).
We set $\, \CC^q := \bigsqcup_{n \geq 0} \CC^q_n \,$, assuming
that $\, \CC^q_0 \,$ is a point.

In a slightly more invariant way, we may think of $ \CC^q $ as the
space of (isomorphism classes of) triples $\,(V,\,X,\,Y)\,$, where
$ V $ is a finite-dimensional complex vector space, and
$\,X,\,Y\,$ are automorphisms of $V$ satisfying the condition
\begin{equation}
\label{qrone'}
\rk(qXYX^{-1}Y^{-1} - 1_V)= 1\ .
\end{equation}
Now, for each $ n \ge 0 $, there is a natural action on $
\CC_n^q $ by the lattice $ \Z^2 $:
$$
(V,\,X,\,Y) \mapsto (V,\,q^k X,\,q^m Y)\ , \quad (k,\,m) \in \Z^2\ .
$$
We write $\,\bCC_n^q := \CC_n^q/\Z^2\,$ for the corresponding quotient spaces, and
set $\, \bCC^q := \bigsqcup_{n \geq 0} \bCC^q_n \,$.

Now, let $\RR^q := \RR(A_q) $ be the set of isomorphism classes of right
ideals of $A_q$. Our first main result is the following
\begin{theorem}
\la{MTh}
There is a natural bijection $\, \omega:\,\bCC^q \stackrel{\sim}{\to} \RR^q \,$.
\end{theorem}
As in the Weyl algebra case, the bijection $ \omega $ can be described 
quite explicitly; it is induced by the map $\, \CC^q \to \RR^q\,$, 
assigning to $\,(V,\,X,\,Y) \in \CC^q \,$ the class of (isomorphic) fractional
ideals
\begin{gather}
\label{qmx1}
M_x=\det(X-x1_V)\cdot A_q+\kappa^{-1}\det(Y-y1_V)\cdot A_q\,,\\
\label{qmy1} M_y=\det(Y-y1_V)\cdot A_q+\kappa\det(X-x1_V)\cdot
A_q\ ,
\end{gather}
where $\, \kappa \,$ and $\kappa^{-1} $ are given by
the following elements in the quotient field of $A_q$
\begin{gather}\label{qka1}
\kappa=1+j(qY-y1_V)^{-1}(X-x1_V)^{-1}i\,,\\\label{qka2}
\kappa^{-1}=1-j(X-qx1_V)^{-1}(Y-y1_V)^{-1}i\,.
\end{gather}
Theorem~\ref{MTh} thus implies that the right ideals of
$A_q$ (and hence rank one projective $A_q$-modules) are
classified by the conjugacy classes of pairs of matrices
$\,(X,\,Y) \in \GL(V) \times \GL(V)\,$ satisfying
\eqref{qrone'}. Furthermore, every ideal of $ A_q $ is 
isomorphic to one of the form \eqref{qmx1}--\eqref{qmy1},
with pairs $(X,\,Y)$ and $(X',\,Y')$ corresponding to isomorphic 
ideals if and only if $(X',\,Y')\sim (q^kX,\,q^mY)$ for some 
$\,k,\, m \in \Z\,$. As mentioned in the Introduction, this result is similar to 
that for the Weyl algebra $A_1 $, and in the existing literature there are 
several different proofs (see \cite{BW1, BW2, BC, BCE}).
Each of these proofs can be extended to the quantum case. In Section~\ref{S5}, 
we will outline a proof generalizing the arguments of \cite{BC, BCE}: the advantage 
of this approach is that it exhibits an interesting connection between $A_q$ and 
the representation theory of double affine Hecke algebras $\, \H_{1,\, q^{-1/2}}(S_n)\,$, 
which may be of independent interest.

We now discuss some interesting implications of Theorem~\ref{MTh}.

\section{The Picard Group of $ A_q$}
\la{S3}

We begin by introducing notation. Being a Noetherian domain,
the algebra $A_q $ can be embedded into a quotient skew-field, 
which we denote by $Q$. The spaces of
nonzero (Laurent) polynomials in $x$ and $y$ form two Ore subsets
in $A_q $: we write $ \c(x)[\y] \subset Q $, resp. $\,\c(y)[\x]
\subset Q $, for the corresponding localizations. Every element
$\,a \in \c(x)[\y] \,$ can be uniquely written in the form $\,a =
\sum_{m\le i\le n} a_{i}(x)y^i\,$, with $\,a_{i}\in\c(x)\,$ and
$\,a_m,\, a_n\ne 0$. We call $a_n(x) $ the {\it leading
coefficient} of $a$, and the difference $n-m$ the {\it degree} of
$a$.
\blemma \la{Mx}
Every ideal $M$ in $ A_q $ is isomorphic to a fractional ideal
$ M_x $ satisfying

$(1)$ $\ M_x \subset \c(x)[\y]\,$ and $\, M_x \cap \c[\x] \ne \{0\}\,$;

$(2)\ $ all the leading coefficients of elements of $\,M_x\,$ belong to $ \c[\x]\,$;

$(3)$ $\ M_x $ contains an element with a constant leading coefficient.

If $ M_x $ and $ M_x' $ are two fractional ideals of $A_q $, both isomorphic to $M$,
and satisfying $(1)$-$(3)$, then there is a unit $\, u  \in  A_q \,$ such that
$M_x' = u\,M_x$.
\elemma
\bproof
The proof is essentially the same as in the Weyl algebra case (see \cite{BW2}, Lemma~5.1).
First, by \cite{Sta}, Lemma 4.2, every ideal class of $A_q$ contains a
representative $M$ such that $\,M \cap\c[\x]\ne\{0\}\,$. The leading coefficients
of all the elements of $M$ form an ideal in $ \c[\x] $; taking a generator
$ p(x) \in \c[\x] $ of this ideal, we set $\, M_x := p^{-1} M $. It is easy to see that
$ M_x $ thus defined satisfies the properties $(1)$-$(3)$.

Now, if $ M_x' $ is another (fractional) ideal isomorphic to $ M $, we have
$\,M_x' = \gamma \,M_x \,$ for some $ \gamma \in Q $. If both $M_x$ and $ M_x'$
satisfy $(1)$, then $\, \gamma \,$ must be a unit in  $\, \c(x)[\y] \,$ and hence
has the form $\,\gamma = f(x) y^{k}\,$, with $\, f(x) \in \c(x)\setminus\{0\} \,$ and
$ k \in \Z $. Property $(2)$ forces $ f(x) $ to be polynomial, i.~e. $\, f(x) \in \c[\x]\,$, and
then $(3)$ implies that $\, f(x) = \alpha \,x^m \,$ for some $\, \alpha \in \c^* $ and $ m \in \Z $.
Thus $\,\gamma = \alpha\,x^m y^k \,$ is a unit in $A_q$.
\eproof

Given any (fractional) ideal $\, M \subset Q \,$ of $A_q$, its endomorphism ring is naturally
identified with a subring of $ Q \,$:
$$
\End_A(M) = \{\gamma \in Q\ :\ \gamma M \subseteq M\} \ .
$$
This yields a group embedding $\, \Aut_A(M)  \into
Q^* \,$, whose image we denote by $\, U(M) $. Thus $\,U(M) =
\{\gamma \in Q^*\ :\ \gamma M = M\}\,$. In particular, if $ M =
A_q $, we  get the group of units $ U(A_q) $ of  $A_q $ 
canonically embedded in $ Q^* $. 
\blemma \la{Ux} Let $ M_x $ be a fractional ideal of $A_q $
satisfying the properties $(1)$-$(3)$ of Lemma~\ref{Mx}. Then $\,
U(M_x) \subseteq U(A_q) \,$. Moreover, if $\, M_x' \cong M_x \,$
is another representative satisfying $(1)$-$(3)$, we have $\, U(M_x')
= U(M_x) \,$ in $\,U(A_q) $. Thus the subgroup $\, U(M_x)\subseteq
U(A_q) $ depends only on the class of $ M_x $ in $ \RR^q $.
\elemma
\bproof The fact that $\, U(M_x) \subseteq U(A_q) \,$ is immediate
from (the last statement of) Lemma~\ref{Mx}. It implies that every
element in $ U(M_x)$ has the form $\, \gamma = \beta\,x^k y^m \,$,
where $ \beta \in \c^*$, $\,k,\,m \in \Z \,$. Using the
commutation relations of $A_q$, it is easy to see that $\,
u\,\gamma\,u^{-1} = \tilde{\beta}\, x^k y^m \,$ for any $\, u\in
U(A_q)\,$. Hence $\, u\,U(M_x)\,u^{-1} = U(M_x)\,$, which proves
the second claim of the lemma. \eproof

\vspace{1ex}

\begin{remark}
Reversing the roles of $x$ and $y$ in the above lemmas, we obtain
another set of distinguished representatives  $\, M_y \subset
\c(y)[\x] \,$ for any given class in $ \RR^q $. The ideals
\eqref{qmx1} and \eqref{qmy1} appearing in Theorem~\ref{MTh} are
examples of such representatives.
\end{remark}

\vspace{1ex}

The next result is an important consequence of Theorem~\ref{MTh}.

\begin{theorem}
\la{unit}
\la{nonc}
Let $M$ be a noncyclic right ideal of $A_q $. Then $\, \Aut_A(M) \cong \c^* $.
\end{theorem}
\bproof
First of all, note that $\,\c^* \subseteq U(M) \subset Q\,$ for any ideal $M$. 
Now, suppose that $\Aut_A(M) \ncong \c^*$. Choose a representative $M_x$ in the 
isomorphism class of $M$ satisfying the conditions $(1)$-$(3)$ of Lemma~\ref{Mx}. 
Since 
$\,\Aut_A(M_x) \cong \Aut_A(M) \ncong {\mathbb C}^*$, we have 
${\mathbb C}^* \subsetneq U(M_x)$.

By Theorem~\ref{MTh}, we may assume that $ M_x $ has the form
\eqref{qmx1} and take $ M_y $ to be the second representative
\eqref{qmy1}. The ideals $M_x$ and $ M_y$ are related by $\,M_y =
\kappa\,M_x\,$, where $\,\kappa = \kappa(x,\,y) \in Q \,$ is
defined in \eqref{qka1}.  It follows that
$$
\gamma \in U(M_x) \quad \Leftrightarrow \quad \kappa\,\gamma \,\kappa^{-1} \in U(M_y)\ .
$$
Now, choose $\gamma \in U(M_x) \setminus {\mathbb C}^*$. Without loss of generality, we can write
$\gamma=x^ky^l$ for some $(k,l) \neq (0,0)$. Then $\,\kappa\, \gamma \, \kappa^{-1}=\alpha x^c y^d\,$ for some 
$\,(c,d) \neq (0,0) \,$ and $\, \alpha \in {\mathbb C}^*$. Whence 
$\,\kappa \gamma^m \kappa^{-1} = {( \alpha x^c y^d)}^m\,$ for every $m \in \mathbb Z$.
It follows that, for any $m \in \mathbb Z$, there exists $\,\alpha_m \in \mathbb C^*\,$ 
such that
\begin{equation} 
\label{a1} 
\gamma^{-m} \kappa \gamma^m \kappa^{-1} = \alpha_m x^{m(c-k)}y^{n(d-l)} \text{ .} 
\end{equation}
If $\,(k,l) \neq (c,d)\,$, we may choose $m$ so that $\,\text{max}\{m(c-k),\,m(d-l)\} > 0\,$. For such $\,m\,$, the equation 
\eqref{a1} clearly does not hold. Therefore we have $\,(c,d)=(k,l)\,$. Further, by comparing the constant terms 
of the Laurent series expansions on both sides of \eqref{a1} (for $m=1$), we find that $\alpha_1=1$. Using the
commutation relations in $ Q $, we may therefore rewrite the equation \eqref{a1} for $m=1$ in the
form
\begin{equation}
\la{eka1} \kappa( x,\,y) = \kappa(q^l x,\,q^{-k} y)\ .
\end{equation}
Expanding now $ \kappa $ into the Laurent series
$$
\kappa(x,\,y) = 1 + \sum_{s,r \ge 0}\,a_{sr}\, y^{-s-1} x^{-r-1}\ , \quad a_{sr} = q^s\,j\,Y^s X^r i\ ,
$$
and substituting it into \eqref{eka1}, we get
\begin{equation}
\la{aa1}
a_{sr}\,\left( 1 - q^{k(s+1) - l(r+1)} \right) = 0\ ,\quad \forall\ r,\,s \in \Z_{+}\ .
\end{equation}

It follows that
$\, a_{rs} = 0 \,$ for all $\, r,\,s\,$, such that $\label{a2} k(s+1) \, \neq \, l(r+1)$.
Now, since $\,X \in \GL_n(\mathbb C)\,$, the characteristic polynomial of $X$ has a nonzero constant term. 
Hence $\,{1}=\sum_{p=1}^{n} c_pX^p\,$ for some $\,c_p \in \mathbb C$. Suppose that 
$\,s,\,r \in \Z_{+}$ satisfy $k(s+1)=l(r+1)$.  Then $k$ and $l$ are both nonzero. Since 
$k(s+1) \neq l(r+p+1)$ for any $p>0$, we have
\begin{equation} \label{a3} 
a_{sr}=q^s \, j\, Y^s X^r i \, = 
\,\sum_{p=1}^{n} c_p\, q^s\, j \,Y^s X^{r+p} i\,  = \sum_{p=1}^{n} c_p \,a_{s \, r+p} =0 \text{ .}
\end{equation}
This shows that $a_{sr}=0$ for all $s,\,r \in \Z_{+}$. Thus $ \kappa = 1 $ and $\, M_x = M_y = A_q \,$.
It follows that $\, M \cong M_x \,$ is a cyclic ideal.
\eproof

Now, using Theorem~\ref{MTh}, we compute the Picard group of the
algebra $ A_q $. Recall that the elements of $ \Pic(A_q) $
are the isomorphism classes of invertible bimodules of $ A_q $,
which are symmetric over $ \c $. There is a natural group homomorphism
$\,\Omega:\,\Aut(A_q) \to \Pic(A_q)\,$, taking $\, \sigma \in
\Aut(A_q) \,$ to the class of the bimodule $ (A_q)^{\sigma} $,
which is isomorphic to $A_q$ as a right module, with left action
of $A_q$ twisted by $ \sigma^{-1} $. By \cite{F}, Theorem~1, the
kernel of this homomorphism is precisely the group $\,\mbox{\rm
Inn}(A_q)\,$ of inner automorphisms of $A_q$, while $ \im(\Omega)
$ consists of those invertible $A_q$-bimodules that are cyclic as
right modules. Since $ A_q $ is a domain, an invertible bimodule
over $A_q$ is just a right ideal $M$ of $A_q $ such that $
\End_{A_q}(M) \cong A_q $ as $\c$-algebras. This last condition
implies that $ \Aut_{A_q}(M) \cong U(A_q) $, so by
Theorem~\ref{nonc}, $ M $ is indeed cyclic. Thus $\,\Omega\,$
is surjective, and we have 

\bprop\la{sexs}
The canonical sequence of groups
\begin{equation}\la{sexs1}
1 \to \mbox{\rm Inn}(A_q) \to  \Aut(A_q) \stackrel{\Omega}{\longrightarrow} \Pic(A_q) \to 1
\end{equation}
is exact.
\eprop
With Proposition~\ref{sexs}, the problem of computing $ \Pic(A_q)
$ reduces to describing the automorphisms of $ A_q $. Let $\,
(\c^*)^2 $ denote the direct product of multiplicative groups of
$\c$. We define an action of $ \SL_2(\Z) $ on $ (\c^*)^2 $ by
\begin{equation}\la{actt}
g\, : \, (\alpha,\,\beta) \ \mapsto (\alpha^a \beta^b,\, \alpha^c \beta^d)\ ,
\quad
g = \begin{pmatrix}
a & b \\
c & d
\end{pmatrix} \in \SL_2(\Z)\ ,
\end{equation}
and form the semidirect product $\ (\c^*)^2 \rtimes \SL_2(\Z) \,$ relative to this action.
The following result is probably well known (see, e.g., \cite{AD1, AD2}).

\blemma
\la{Aut}
$\, \Aut(A_q) \cong  (\c^*)^2 \rtimes \SL_2(\Z) \,$.
\elemma
\bproof We first recall the well-known presentation of $ \SL_2(\Z) $ as a quotient of
the braid group $ B_2 \,$:
\begin{equation}\la{braid}
\SL_2(\Z) = \langle g_1\, ,g_2\, | \text{ } \, g_1g_2g_1=g_2g_1g_2,\,\, \, (g_1g_2)^6= 1 \, \rangle .
\end{equation}
The braid generators of $\,g_1$ and $ \,g_2 $ correspond under \eqref{braid} to the following matrices
$$
g_1 = \begin{pmatrix}
1 &  1 \\
0 & 1
\end{pmatrix} \quad , \quad
g_2 = \begin{pmatrix}
1 & 0\\
-1 & 1
\end{pmatrix}\ .
$$
Now, using this presentation, it is easy to check that
$$
g_1\,: \, (x,\,y)  \mapsto (yx,\, y)\quad , \quad g_2\,: (x,\,y)  \mapsto (x,\, yx^{-1})\ .
$$
extends to a well-defined group homomorphism: $\,\SL_2(\Z) \to \Aut(A_q) \,$ of the following
form
$$
g\,:\ (x,\,y)  \mapsto (\alpha_g\,y^b x^a,\, \beta_g\,y^d x^c) \ ,
$$
where $ g \in \SL_2(\Z) $ as in \eqref{actt} and $\, \alpha_g, \beta_g \in \c^* \,$ are some constants depending
on $ g $. On the other hand, there is an obvious homomorphism $\, (\c^*)^2 \to \Aut(A_q) \,$, mapping
$\,(x,\,y) \mapsto (\alpha x, \,\beta y)\,$ for $\,(\alpha,\,\beta) \in (\c^*)^2 $. These two homomorphisms
fit together giving
\begin{equation}
\la{act2}
(\c^*)^2 \rtimes \SL_2(\Z) \to \Aut(A_q)\ ,
\end{equation}
which is easily seen to be injective.

On the other hand, any element $\,\sigma \in \Aut(A_q) \,$ takes units to units: in
particular, it maps the generators $\, (x,\,y) \,$ to elements of
the form $\,(\alpha\,y^b x^a,\,\beta\, y^d x^c)\,$, with
$\,\alpha,\,\beta \in c^*\,$ and $\,a,\,b,\,c,\,d \in \Z\,$. The
invertiblity of $ \sigma $ means that $\,ad-bc= \pm 1\,$ and the
relation $\,\sigma(xyx^{-1}y^{-1}) = q\,$ ensures that $\,ad-bc=
1\,$. This proves that \eqref{act2} is surjective. 
\eproof

Under the identification of Lemma~\ref{Aut}, the inner
automorphisms of $ A_q $ correspond to the elements $\, (q^{n}
x,\,q^{m} y;\,1) \in (\c^*)^2 \rtimes \SL_2(\Z) \,$, with $\,n,\,m
\in \Z\,$. Indeed, one can compute easily that the canonical
projection $\, U(A_q) \onto \mbox{\rm Inn}(A_q) \,$, $\,u \mapsto
\mbox{\rm Ad}(u)\, $, is given on generators by $\,\mbox{\rm
Ad}(u):\, (x,\,y) \mapsto (q^{-b}x,\,q^{a}y)\,$, where $\,u =
\alpha \,x^a y^b \in U(A_q)\,$. Thus, we arrive at the following
theorem, which is the main result of this paper.
\bthm
\la{Pic}
With identification of Lemma~\ref{Aut}, the canonical map $ \Omega $ induces
an isomorphism of groups
\begin{equation}
\la{picc}
\Pic(A_q) \cong (\c^*\!/\Z)^2 \rtimes \SL_2(\Z) \ ,
\end{equation}
where $ \Z $ is identified with the cyclic subgroup of $\, \c^* $
generated by $ q $.
\ethm

We end this section with a side observation. Assume that $\, |q| \ne 1 \,$.
Regarding $ \c^* $ as complex analytic space, we can then identify the quotient $\,\c^*\!/{\Z}\,$ with
a (smooth) elliptic curve $ X $. Let $ \D^b(X) $ denote the bounded derived category of coherent sheaves on
$ X $, and let $\, \Auteq\,\D^b(X)\,$ be the group of (exact) auto-equivalences of this category. Comparing
our Theorem~\ref{Pic} with results of \cite{Or, ST}, we get
\begin{corollary}
\la{Pell}
There is a natural group isomorphism
\begin{equation}
\la{g}
\gamma:\ \Pic(A_q) \stackrel{\sim}{\rightarrow} \Auteq(\D^b(X))/{\Z} \ .
\end{equation}
where $\, \Z \subset \Auteq(\D^b(X)) \,$ corresponds to the subgroup of translation functors on $ \D^b(X) $.
\end{corollary}
We will briefly explain how to construct the isomorphism $ \gamma $.
Let $ x_0 \in X $ be the point corresponding to the image of $\, 1 \in \c^* \,$ under the canonical projection
$\,\c^* \onto X \,$. By \cite{Or} (see also \cite{ST}), the group $\, \Auteq \,\D^b(X) \,$ is
then generated by the functors
\begin{equation}
\la{auteq}
T_{\mathcal O},\, T_{{\mathcal O}_{x_0}}\, \text{ and } R_x,\,L_x \ (x \in X)\ ,
\end{equation}
where $\,L_x\,$ is induced by tensoring with the line bundle $\,\O(x-x_0)$, $\,R_x $ is the pull-back
via the automorphism $\,X \to X\,$, $\,z \mapsto x \cdot z\,$, and
$T_{\mathcal E}$ denotes the Fourier-Mukai transform with kernel
$$
{\tt Cone}({\mathcal E}^{\vee} \boxtimes {\mathcal E} \rightarrow
\Delta_*{\mathcal O}_X) \ , \qquad {\mathcal E} \in  \D^b(X)\ .
$$
Now notice that there is the obvious homomorphism
\begin{equation}
\la{g2}
\gamma_1: \, X \times X \rightarrow \text{Auteq}(\D^b(X))/{\Z}\ ,
\quad  (x,y) \mapsto L_x\,R_y \ .
\end{equation}
More interestingly, using the known relations between the functors \eqref{auteq} (see \cite{ST}, Section~3d), one can 
easily check that  $\,g_1 \mapsto T_{\mathcal O_{x_0}} \,$ and $\,g_2 \mapsto T_{{\mathcal O}}\,$ extend to 
a well-defined homomorphism
\begin{equation}
\la{g1}
\gamma_2: \ \SL_2(\Z) \rightarrow \text{Auteq}(\D^b(X))/{\Z}\ .
\end{equation}
Another direct calculation shows that the maps \eqref{g2} and \eqref{g1}  agree with each other extending to the semidirect product 
$\,(X \times X) \rtimes \SL_2(\Z) \,$. By Theorem~\ref{Pic}, this defines the desired isomorphism \eqref{g}. It would be very 
interesting to find a conceptual explanation for this isomorphism (cf. \cite{ST}, Remark~1.5). 
The results of \cite{BaEG} as well as recent papers \cite{SV} and  \cite{Po} suggest that the existence of $ \gamma $ may not be
a mere coincidence.

\section{Morita classification}
\la{S4}

We begin by classifying the algebras $ A_q $ up to Morita equivalence within the family
$\,\{A_q\}\,$ and then consider the general case.
\bthm
\la{Th2} Under the assumption \eqref{nu}, the algebras $ A_q
$ and $ A_{q'} $ are Morita equivalent if and only if they are
isomorphic, i.~e. if and only if $\, q' = q \,$ or $\, q' = q^{-1}
$.
\ethm
\bproof
The fact that $\, A_q \cong A_{q'} \,$ $\,\Leftrightarrow
\,$ $\,q' = q^{\pm 1} $ follows easily from the defining relations
of $A_q$ and $A_{q'}$ (see \cite{J}, Theorem~1.3). Now, assume
that $ A_{q'} $ and $A_q $ are Morita equivalent. Then  $\, A_{q'}
\cong \End_{A_q}(M) \,$ for some f. g. projective right module
$M$. Since $A_{q'} $ is a domain, $ M $ is isomorphic to a right
ideal in $ A_q $ (see, e.g., \cite{BEG}, Lemma~3). Since 
$\End_{A_q}(M) \cong A_{q'}$, $\Aut_A(M) \ncong {\mathbb C}^*$. 
By Theorem~\ref{unit}, $M \cong A_q$. Therefore, $A_{q'} \cong A_q$ 
which implies that $\, q'= q^{\pm 1}$. 
\eproof

\begin{remark}
Theorem~\ref{Th2} was first proven in~\cite{RS} by a different method.
If $\, q = e^{2\pi i \theta}\,$ and $\,q' = e^{2 \pi i \theta'} \,$, with $ \theta,\,\theta' \in
\R\!\setminus\!\Q $, then Theorem~\ref{Th2} says that $ A_{\theta}  $ and $ A_{\theta'}  $ are equivalent if
and only if $\,\theta' \pm \theta \in \Z \,$. This result should be compared with a well-known Morita
classification of smooth noncommutative tori $\, \ms{A}_{\theta} \,$ (see \cite{R1}): in that
case, $ \ms{A}_{\theta} $ and $ \ms{A}_{\theta'} $ are (strongly) Morita equivalent if and only if
$ \theta $ and $ \theta' $ are in the same orbit of $ \GL_2(\Z) $ acting on $\, \R\!\setminus\!\Q \,$ by
fractional linear transformations. Thus, unlike in the smooth case, there are no interesting Morita 
equivalences between the algebras
$A_{q}$ for different values of $ q $. However, in the
Morita class of each $ A_{q} $ there are many non-isomoprhic algebras corresponding
to different orbits of $ \Pic(A_q) $ in $ \RR $. We will classify these algebras below,
using Theorem~\ref{MTh}.
\end{remark}

\vspace{1ex}

We now describe an action of $ \Pic(A_q) $ on the (reduced) Calogero-Moser spaces
$ \bCC^q_n $. In Section~\ref{S5}, we will show that this action comes from a natural
action of the braid group $ B_2 $ on the double affine Hecke algebra $ \H_{1,\tau}(S_n) $
constructed by Cherednik in \cite{Che}.
We begin by defining an action of $ \Aut(A_q)$ on $ \CC_n^q $. With identification
of Lemma~\ref{Aut}, it suffices to construct two compatible group homomorphisms
\begin{equation}\la{fg}
f_1:\ (\c^*)^2 \to \Aut(\CC_n^q)\quad  \mbox{and} \quad f_2:\ \SL_2(\Z)
\to \Aut(\CC_n^q)\ ,
\end{equation}
where $\,\Aut(\CC_n^q)\,$ denotes the group of regular (algebraic)
automorphisms of $ \CC_n^q $. First, we let
$$
f_1(\alpha,\,\beta)\, :\ (X,\,Y) \mapsto (\alpha^{-1} X,\,\beta^{-1}\,Y)\ ,
\quad (\alpha,\,\beta)\in (\c^*)^2\ .
$$
Next, to define the second homomorphism we will use the presentation \eqref{braid}: on the braid generators,
we define $ f_2 $ by
$$
g_1:\,(X, Y) \mapsto (Y^{-1}X, Y)\ ,\quad g_2:\, (X, Y) \mapsto (X, YX)\ .
$$
A direct calculation then shows that $\, f_2(g_1g_2g_1) = f_2(g_2g_1g_2) \,$ and
$\,f_2(g_1g_2)^6 = 1 \,$ in $\, \Aut(\CC_n^q) \,$. Hence, this assignment
extends to a well-defined homomorphism $\,f_2:\, \SL_2(\Z) \to \Aut(\CC_n^q)\,$.
It is also easy to check that $ f_1 $ and $ f_2 $ are compatible in the sense that
\begin{equation}\la{acf}
f := (f_1,\,f_2)\,:\ (\c^*)^2 \rtimes \SL_2(\Z) \to \Aut(\CC_n^q)\ .
\end{equation}
is a group homomorphism. Now, with identification of Theorem~\ref{Pic}, we see that $ f $ induces
\begin{equation}\la{acf1}
\bar{f}\, :\ \Pic(A_q) \to \Aut(\bCC_n^q)\ ,
\end{equation}
which defines an action of $\Pic(A_q)$ on each of the spaces $
\bCC_n^q$ and hence on their disjoint union $\, \bCC^q $. On the other hand, $ \Pic(A_q)$ acts
naturally on the space of ideal classes $ \RR^q = \RR(A_q) $. With these actions, we have
\bprop
\la{Teqv}
The map $\, \omega:\,\bCC^q \to \RR^q \,$ of Theorem~\ref{MTh} is equivariant under $ \Pic(A_q) $.
\eprop
We will prove Proposition~\ref{Teqv} in Section~\ref{S5}.  Its meaning becomes clear from the following
theorem which gives a geometric classification of algebras Morita equivalent to $ A_q $.
\bthm
\la{Mor}
There is a natural bijection between the orbits of $ \Pic(A_q) $ in $ \bCC^q $
and the isomorphism classes of domains Morita equivalent to $ A_q $.
\ethm
\bproof
The map $\, \omega:\,\bCC^q \to \RR^q \,$ assigns to a point in $ \bCC^q $ an isomorphism class $ [M] $ of ideals
in $ A_q $. Choosing a representative $ M $ in such a class and taking its endomorphism ring $\,\End_{A_q}(M)\,$
yields a domain Morita equivalent to $A_q $. Now, a well-known result in ring theory (see, e.g, \cite{F}, Theorem~1)
says that the rings $\,\End_{A_q}(M)\,$ and $\,\End_{A_q}(M')\,$ are isomorphic
iff the corresponding classes $ [M] $ and $ [M'] $ are in the same orbit of $ \Pic(A_q) $ in $ \RR_q $. Since,
by Morita Theorem, every domain equivalent to $A_q$ is of the form $ \End_{A_q}(M)$, Theorem~\ref{Mor} follows
immediately from Theorem~\ref{MTh} and Proposition~\ref{Teqv}.
\eproof

\begin{remark}
If $ D $ is an algebra Morita equivalent to $ A_q $, which is {\it not} a domain,
then $\,D \cong {\mathbb M}_r(A_q)\,$ for some $ r \ge 2 $. This follows from the fact that
all projective modules over $ A_q $ of rank $ r \ge 2 $ are free.
\end{remark}

\vspace{1ex}

In general, the domains Morita equivalent to $ A_q $ seem to have a complicated stucture; in particular, they are
not easy to describe in terms of generators and relations like $ A_q $. However, their automorphism groups
can be described geometrically, in terms of the action of $\Pic(A_q)$ on $\bCC^q $. Precisely, we have
the following consequence of Theorem~\ref{MTh} and Theorem~\ref{unit}.
\begin{proposition}
\la{autend}
Let $ M $ be a noncyclic right ideal of $A_q $, and let $ E = \End_{A_q}(M) $ denote its endmorphism algebra. 
Then $\, \Aut_\c(E) \,$ can be naturally identified
with a subgroup of $ \Pic(A_q) $, which is isomorphic to the stabilizer of the point $\,\omega^{-1}[M]\,$
under the action $\,\Pic(A_q)\,$ on $\bCC^q $.
\end{proposition}
\begin{proof}
Put $\,A:=A_q\,$. Then $M$ is naturally a $E$-$A$-bimodule.
Let $ M^* := \Hom_{A}(M,\,A) $ be its dual which is an $A$-$E$-bimodule. By Morita theorem, we
have $\,M^* \otimes_E M \cong A\,$ as $A$-bimodules and $\, M \otimes_A M^* \cong E \,$ as $E$-bimodules.
Now, consider the canonical sequence of groups
$$
1 \rightarrow \text{Inn}(E) \rightarrow
\text{Aut}(E) \stackrel{\Omega}{\longrightarrow} \Pic(E) \stackrel{\Psi}{\to} \Pic(A) \ ,
$$
where the last map is an isomorphism groups given by $\,P \mapsto M^* \otimes_E P\otimes_E M\,$.
The image of $\Omega $ consists of those invertible $E$-bimodules that are cyclic as right $E$-modules.
Given such an $E$-bimodule $L$, we have
$$
M \otimes_A M^* \otimes_E L \otimes_E M \cong E
\otimes_E L \otimes_E M \cong L \otimes_E M \cong M
$$
as right$A$-modules. Hence, $\,\Psi(L) = M^* \otimes_E L \otimes_E M$ is in the stabilizer
of $[M]$ under the action of $\Pic(A)$ on $ \RR^q $. Conversely, if $P$ is an
invertible $A$-bimodule such that $\,M \otimes_A P \cong M\,$, then
$$
M \otimes_A P \otimes_A M^* \cong M \otimes_A M^* \cong E
$$
as right $E $ modules. Thus, $\, \Psi^{-1}(P) \cong M \otimes_A P \otimes_A M^*\,$ is in the image of
$\Omega$. It follows now from Theorem ~\ref{MTh} that $\im(\Omega)$ is isomorphic to the stabilizer
of $ \omega^{-1}[M] $ under the $\Pic(A)$-action on $ \bCC^q $. On the other hand, by Theorem~\ref{unit},
$\,\text{Inn}(E)\,$ is trivial. Hence $\,\Aut(E) \cong \im(\Omega)\,$. This finishes the proof of the
proposition.
\end{proof}

\section{Double Affine Hecke Algebras and the Calogero-Moser Correspondence}
\la{S5}
In this section, we describe a relation between $ A_q $ and the double affine Hecke algebras 
$\, H_{q,n} := \H_{1,\,q^{-1/2}}(S_n) \,$. Our construction generalizes (and simplifies) 
the results of \cite{BCE}, where a similar relation between $A_1$ and the rational Cherednik
algebra has been studied. A key role in this construction is played by a multiplicative 
version of the deformed preprojective algebra of a quiver introduced in \cite{CBS}. We draw reader's attention 
to the fact that we use a more general form of these algebras in which weights are assigned not
only to the vertices but also to the edges of the quiver.
To simplify exposition we omit most routine calculations, especially those ones parallel to \cite{BCE}. 
However, we will give some details in the proof of Proposition~\ref{Teqv}, since the idea of this proof 
has not been used in the case of the Weyl algebra.
\subsection{DAHA and the Calogero-Moser spaces} 
We recall the presentation of the double affine Hecke algebra $\, H_{q,n} \,$  (see \cite{Che}).

\vspace{1ex}

\noindent
{\rm Generators}: 
\begin{eqnarray*}
X_1^{\pm 1}\ ,\ X_2^{\pm 1}\ , \ \ldots\ ,\ X_n^{\pm 1}\ ;\quad T_1\ , \ T_2\ ,\ \ldots\ ,\ T_{n-1}\ ;\ \pi \ .
\end{eqnarray*}

\noindent
{\rm Relations}:
\begin{eqnarray*}
&&X_i\,X_j = X_j\,X_i\ , \quad T_i\,X_i\,T_i = X_{i+1} \ , \quad \pi\, X_{i} = X_{i+1}\,\pi\ , \quad 1 \, \leq i \, \leq n-1\ , \\*[1ex]
&&T_i\,T_{i+1}\,T_i = T_{i+1}\,T_i\,T_{i+1}\ , \quad \pi \,T_i = T_{i+1}\,\pi\ ,\quad 1 \leq i \leq n-2\  ,\\*[1ex]
&&T_iX_j = X_jT_i\quad (j-i \neq 0,1)\ ,\quad \text{[}T_i,T_j\text{]}=0 \quad (|i-j|>1)\ , \\*[1ex]
&&\pi\, X_n = X_1\, \pi \quad , \quad \pi\, T_i = T_{i+1}\,\pi\quad , \quad  \pi^n\, T_i= T_i\,\pi^n\ , \\*[1ex]
&&(T_i- q^{-1/2})\,(T_i+ q^{-1/2}) = 0\ .
\end{eqnarray*}
Further, following \cite{Che}, we introduce $\,n\,$ pairwise commuting elements in $\,H_{q,n} \,$:
$$
Y_i := T_i\,T_{i+1}\,\ldots\,T_{n-1}\,\pi^{-1}\,T_1^{-1}\,\ldots\,T_{i-1}^{-1}\ ,\quad i = 1, \,2,\,\ldots\,,\, n \ ,
$$
satisfying the relations
$$
T_iY_{i+1}T_i = Y_i \quad \text{and} \quad T_iY_j = Y_jT_i \ ,\quad j-i \neq 0,\, 1\ .
$$
The algebra $ H_{q,n} $ contains a copy of the finite Hecke algebra $\,\mathcal{H}_{q^{-1/2}}(S_n)\,$ of
the group $ S_n $, which, in turn, contains the idempotents
\begin{equation}
\la{ide}
\varepsilon := \frac{\sum_{w \in S_n} \tau^{l(w)}T_w}{\sum_{w \in
S_n} \tau^{2l(w)}} \quad  \text{and}\quad  
\varepsilon':= \frac{\sum_{w \in
S_{n-1}} \tau^{l(w)}T_w}{\sum_{w \in S_{n-1}} \tau^{2l(w)}} \ .
\end{equation}
Here $\,S_{n-1} \,$ is regarded as the subgroup of $ S_n $ which consists of permutations of 
$\,\{1,\,2,\, \ldots\, ,\,n\}$ fixing the first element. For $\,w \in S_n\,$, we write
$\,T_w := T_{i_1}...T_{i_{l(w)}}\,$, where $\,w=s_{i_1}\,\ldots\,s_{i_{l(w)}}\,$ is a reduced expression 
of $w$ and $\,s_i:=(i,i+1) \in S_n\,$.

The following proposition is analogous to \cite{EG}, Theorem~1.23 and Theorem~1.24, in the case 
of the rational Cherednik algebra.
\begin{proposition}[see \cite{O}, Theorem~5.1 and Theorem~6.1] 
\la{oblom} \hfill

$(1)$\ The algebra $\,\varepsilon \, H_{q,n}\, \varepsilon\,$ is Morita equivalent to $\, H_{q,n}\,$, the equivalence
$\,\Mod(H_{q,n}) \to \Mod(\varepsilon \, H_{q,n}\, \varepsilon)\,$ being the canonical functor $\,\varepsilon:\,
M \mapsto \varepsilon\,M \,$.

$(2)$ $\,\varepsilon \, H_{q,n}\, \varepsilon\,$ is a commutative algebra isomorphic to the coordinate 
ring of $\, \CC_n^q $.
\end{proposition}
Next, we introduce a multiplicative version of the deformed preprojective algebra of a quiver $Q$,
due to Crawley-Boevey and Shaw \cite{CBS} (see also \cite{vdB}). Our definition is slightly more general than 
that of \cite{CBS} as we assign weights to both the vertices and the arrows of $Q$ (see Remark below). 
\subsection{The (generalized) multiplicative preprojective algebras} Let $ Q $ be a quiver with vertex set $I$.
Let $ \bar{Q} $ be the double of $ Q $ obtained by adjoining a reverse arrow $ a^* $ to each arrow $ a \in Q $.
As in \cite{CBS}, we extend $\,a \mapsto a^* \,$ to an involution on $ \bar{Q} $ by letting
$\, (a^*)^* = a \,$, and define the function $\,\epsilon: \bar{Q} \to \{\pm 1\} \,$ by $\,\epsilon(a) = 1 \,$
if $ a \in Q $ and $\,\epsilon(a) = -1 \,$ if $ a^* \in Q $. Next, we choose two sets of parameters (weights):
$\,\{q_v\}_{v \in I} \,$ and $\,\{\hbar_a\}_{a\in \bar{Q}}\,$ with the assumption that $\, \hbar_{a^*} = \hbar_a \,$ 
for all $\, a \in Q \,$. The {\it multiplicative preprojective algebra} $\,\Lambda^{q,\hbar}(Q)\,$ is now defined by the 
algebra homomorphism $\,\c\bar{Q} \to \Lambda^{q,\hbar}(Q)\,$, which is {\bf universal} among all algebra homomorphisms
$\,\c\bar{Q} \to R\,$, satisfying the properties
\begin{equation*}
\la{pp1}
a a^* + \hbar_a \ \mbox{is a unit in} \ R\ \mbox{for all}\ a \in \bar{Q}\ ,
\end{equation*}
\begin{equation*}
\la{pp2}
\prod_{a \in \bar{Q}} (a a^* + \hbar_a)^{\epsilon(a)} = \sum_{v \in I}\,q_v\,e_v\quad \mbox{in}\ R\ .
\end{equation*}

\begin{remark}
The original definition of multiplicative preprojective algebras (see \cite{CBS}, Definition~1.2) 
corresponds to the choice $\,\hbar_a = 1 \,$ for all $ a \in \bar{Q} $. To see why we need an extension
of this definition consider a quiver $ Q $ which consists of a single vertex $v$ and a single loop $ a $.  
Choosing then $\, q_v = q \,$ and $\, \hbar_a = 0 \,$, we get the algebra $\,\Lambda^{q,\hbar}(Q)\,$
isomorphic to $ A_q $ (with $\, a\, \leftrightarrow x \,$ and $\,a^* \leftrightarrow y\,$).
On the other hand, Example~1.3 in \cite{CBS} shows that if $ \, q_v = q \,$ and $\, \hbar_a = 1 \,$,
then $\,\Lambda^{q,\hbar}(Q) \cong \c \langle x,\,y,\,(1+xy)^{-1} \rangle/(xy-qyx-1)\,$, which
is a different quantized version of the first Weyl algebra, not isomorphic to $A_q$ (see, e.g., \cite{AD2}).
\end{remark}

\vspace{1ex}

Now, as in \cite{BCE}, Sect.~2.2, we consider the {\it framed} one-loop quiver
$\, Q = Q_\infty \,$ with two vertices, $\,I = \{0,\,\infty\}\,$, and two arrows
$\,i:\, 0 \rightarrow \infty\,$ and $\, X:\,0 \rightarrow 0\,$. Write $\, j:=i^*\,$ and
$\,Y:=X^*\,$ for the reverse arrows in $\,\bar{Q}\,$. For the vertex and arrow weights, we 
take $\,(q_0,\,q_\infty) = (q,\,q^{-n})\,$ and $\,(\hbar_X,\,\hbar_i) = (0,1)\,$, respectively. 
The corresponding algebra $\, \Lambda := \Lambda^{q,\hbar}(Q)\,$ can then be identified with
the quotient of $\,\c\bar{Q}\langle U, V \rangle := \c \bar{Q} \ast 
\c\langle U, V \rangle\,$ modulo the relations
\begin{eqnarray*}
&&  U = eUe\ , \quad V= eVe\ ,\quad XU = UX = e\ , \quad YV=VY=e\ ,\\
&& XY-qYX-qYXij=0\ , \quad ji=(q^{-n}-1)e_{\infty}\ ,
\end{eqnarray*}
where $\,e\,$ and $\,e_\infty\,$ are the idempotents corresponding to the vertices 
$0$ and $ \infty $, respectively. The following lemma clarifies the relation between
this algebra and $A_q $: it is a "multiplicative" analogue of \cite{BCE}, Lemma~3.
\blemma
\la{laq}
$A_q$ is isomorphic to the quotient of $ \Lambda $ by the ideal generated by $ \ei $.
\elemma

\noindent
In fact, the required isomorphism is induced by 
$\,\c\langle \x, \y\rangle \to \Lambda/\langle \ei \rangle\,$,
$\,x \mapsto X \,$, $\,y \mapsto Y\,$.

Next, we explain the relation between $ \Lambda $ and the Cherednik algebra $\, H := 
H_{q,n} $. To this end, we consider the left projective $H$-module 
$\, P:= H \varepsilon' \oplus H \varepsilon \,$, where $\,\varepsilon\,$ and
$\,\varepsilon'\,$ are the idempotents defined in \eqref{ide}. The endomorphism ring 
of $\,P\,$ can be identified with a matrix algebra:
$$
\End_H(P) = \left(
                        \begin{array}{lr}
                          \lldaha & \lrdaha \\
                          \rldaha & \rrdaha \\
                        \end{array}
                      \right)\ .
$$
Using this identification, we can define an algebra map 
$\, \Theta:\, \c\bar{Q}\langle U, V \rangle \to \End_{H}(P)^{\circ}\,$ by
\begin{equation*}
X \mapsto \left(
                           \begin{array}{cc}
                             X_1\varepsilon' & 0\\
                             0 & 0 \\
                           \end{array}
                         \right)\ ,\quad
Y \mapsto \left(
                           \begin{array}{cc}
                             Y_1\varepsilon' & 0\\
                             0 & 0 \\
                           \end{array}
                         \right)\ ,\quad
i \mapsto \left(
                \begin{array}{cc}
                  0 & 0 \\
                  \varepsilon & 0 \\
                \end{array}
              \right)\ ,\quad 
j \mapsto \left(
                \begin{array}{cc}
                  0 & \varepsilon \\
                  0 & 0 \\
                \end{array}
              \right) (q^{-n}-1) .
\end{equation*}
The following proposition shows that $ \Theta $ is a multiplicative analogue of the map 
$\, \Theta^{\rm quiver} \,$ constructed in  \cite{EGGO} (see {\it loc. cit.}, (1.6.3)).
\begin{proposition} 
\la{pbimod} 
The map $ \Theta $ induces an algebra homomorphism $\,\Lambda \to \End_H(P)^{\circ} $.
\end{proposition}
The proof of Proposition~\ref{pbimod} is a routine calculation which we leave to the interested reader. 

\subsection{The Calogero-Moser functor} 
One way to interpret Proposition~\ref{pbimod} is to say that $\,P\,$ is an $H$-$\Lambda$-bimodule, with
right $\Lambda$-module structure defined by $ \Theta $. In combination with Lemma~\ref{laq}, this allows us 
to define the functor
\begin{equation}
\la{cmf}
\CM_n:\ \D^b(\Mod\,H) \to \D^b(\Mod\,\Lambda) \to \D^b(\Mod\,A_q)\ ,\quad V \mapsto (V \otimes_{H} P) \Lotimes_{\Lambda} A_q\ .
\end{equation}
Note that  $ P $ is a projective module on the left, so tensoring with $ P $ over $ H $ is an 
exact functor. 

In view of Proposition~\ref{oblom}, Theorem~\ref{MTh} of Section~\ref{S2} is a consequence of the following result.
\bthm
\la{MTh2}
$(1)$ The functor $ \CM_n $ transforms the simple $H$-modules $($viewed as $0$-complexes in $\D^b(\Mod\,H))$
to rank $1$ projective $A_q$-modules $($located in homological degree $-1)$.

$(2)$ Two simple $H$-modules $\,V = (V;\, X_i,\,T_j,\,\pi)\,$ and $\,V' = (V';\, X_i',\,T_j',\,\pi')\,$
correspond to isomorphic $A_q$-modules if and only if there is $\,(k,\,m) \in \Z^2 \,$ such that
$$
X_i' = q^k X_i\ ,\quad T_j' = T_j\ ,\quad \pi' = q^m \pi\ .
$$
$($We call such $H$-modules \,{\rm equivalent}.$)$

$(3)$ For every rank 1 projective $A_q$-module, there is a unique $\, n \in \N \,$ and a simple module $V$
over $ H_{q,n} $ such that $\,\CM_n(V) \cong M[1]\,$ in $\,\D^b(\Mod\,A_q)\,$.
\ethm
Thus the Calogero-Moser map $ \omega $ of Theorem~\ref{MTh} is induced by the functors $\,\CM_n\,$
`amalgamated' over all $n$. The proof of Theorem~\ref{MTh2} is analogous to \cite{BCE}, Theorem~3:
it is based on \cite{BC}, Theorem~3, and the following key lemma. 

To simplify the notation we set $\,R := \c\langle x^{\pm 1}, y^{\pm 1} \rangle\,$ and denote by $\,\alpha \,$ 
the (surjective) algebra homomorphism $\, R \to e \Lambda e\,$ taking 
$\,x,\, x^{-1},\, y,\, y^{-1}\,$ to $\,X,\, U,\, Y,\, V\,$, 
respectively. Using this homomorphism, we define the linear map
\begin{equation*}
\chi:\ R \to \rrdaha \ ,\quad
r \mapsto \Theta(i\,\alpha(r)\,j) \ .
\end{equation*}
\begin{lemma} 
\la{CM2} Regarding $ P $ as a right $\Lambda$-module, we have

$(1)$\ $\, \Tor^{\Lambda}_k(P,\, A_q) = 0 \,$ for all $\, k \ne 1 $.

$(2)$  For $\, k = 1 $, there is an isomorphism of right $ A_q$-modules
\begin{equation} 
\la{cmfunc1}
\Tor^{\Lambda}_1(P,\, A_q) 
\cong 
\Ker \left(\frac{H \varepsilon \otimes R}{H (\varepsilon
\otimes r w - \chi(r) \otimes 1) R }
\xrightarrow{\ \mu\ } H \varepsilon' \right) \ ,
\end{equation} 
where $\,w = q^{-1} (x^{-1}y^{-1}xy)-1\,$.
\end{lemma}
The notation of Lemma~\ref{CM2} needs some explanation. The multiplication-action map $ \mu $ in \eqref{cmfunc1} 
is induced by $\,\mu(a \varepsilon \otimes r)= a \varepsilon\,\pi(r) \,$,
where $\, a \varepsilon\,$ is viewed as an element in the direct summand of $ P $ (note that 
$ a \varepsilon = a \varepsilon \varepsilon'$). The $A_q$-module structure on the right 
hand side of \eqref{cmfunc1} descends then from the natural right
$R$-module structure on $\Ker\,\mu\,$. 
\bproof
$(1)$ By Proposition~\ref{oblom}$(2)$, $\, \rrdaha \cong \O(\CC_n^q) \,$. Write $\,\O := \O(\CC_n^q)\,$.
Then, for every maximal ideal $ \m \subset  \O $, $\,(\O/\m) \otimes_{\O} (\varepsilon P)\,$ is a right
$\Lambda$-module of dimension $\,n\,$ over $\c$. It follows that $\,(\O/\m) \otimes_{\O} (\varepsilon P)
\otimes_{\Lambda} A_q\,$ is a finite dimensional right $A_q$-module and hence zero. Since 
$\,\varepsilon P \otimes_{\Lambda} A_q\,$ is a coherent $\O$-module,  
this implies that $\, \varepsilon P \otimes_{\Lambda} A_q = 0 \,$ by Nakayama's Lemma. But then,
by Morita equivalence of Proposition~\ref{oblom}$(1)$, $\,P \otimes_{\Lambda} A_q = 0\,$.
To see that the second and higher Tor's vanish we observe that the natural map 
$\,\Lambda e_{\infty} \Lambda \rightarrow \Lambda \,$ provides a left (and right) projective 
resolution of $A_q$, so $\, A_q \,$ has projective dimension $ 1 $ as a left $\Lambda$-module.

The proof of part $(2)$ is similar to the proof of Theorem~1 in \cite{BCE}. We leave it 
to the reader.
\end{proof}

\subsection{Equivariance} 
We now turn to the proof of Proposition~\ref{Teqv}. Recall that 
$\,\SL_2(\Z)\,$  is a quotient of the braid group $B_2$. Let $\,\phi:\,B_2 \rightarrow \SL_2(\Z)\,$ 
denote the corresponding projection
$$
g_1 \mapsto \left(
                   \begin{array}{cc}
                     1 & 1 \\
                     0 & 1 \\
                   \end{array}
                 \right) \ , \qquad 
g_2 \mapsto \left(
                                                   \begin{array}{cc}
                                                     1 & 0 \\
                                                     -1 & 1 \\
                                                   \end{array}
                                                 \right)\ .
$$
The $\SL_2(\Z)$-action on $\,(\c^*)^2\,$ defined by $\,\eqref{actt}\,$ induces
(via $\phi$) a $B_2$-action on $(\c^*)^2$. We write $\,G := (\c^*)^2 \rtimes B_2\,$ 
for the corresponding semidirect.
\begin{lemma}[cf. \cite{Che}] 
\la{gpact1}
The following assignment extends to a well-defined group homomorphism 
$\,\Phi_{H}: \, G \to \Aut_{\c}(H_{q,n}) \,$:
\begin{eqnarray*}
(\alpha,\beta; \text{id}) &\mapsto& \nu_{\alpha,\beta}:= (X_i \mapsto \alpha X_i\,\,,\,\,Y_i 
\mapsto \beta Y_i \,\,,\,\,T_i \mapsto T_i)\ ,  \\
(1,1;g_1) &\mapsto& \tau:= (X_i \mapsto Y_iX_i\,\,,\,\,Y_i \mapsto Y_i \,\,,\,\, T_i \mapsto T_i)\ ,\\
(1,1;g_2) &\mapsto& \theta:= (X_i \mapsto X_i\,\,,\,\,Y_i \mapsto X_i^{-1}Y_i \,\,,\,\, T_i \mapsto T_i) \ .
\end{eqnarray*}
\end{lemma}
\begin{proof}
It is obvious that $\,\nu_{\alpha,\beta} \in \Aut_{\c}(H_{q,n})\,$. The fact that  $\,\theta\,$ and
$\,\tau\,$ are also automorphisms, satisfying the braid
relation $\,\theta \tau \theta =\tau \theta \tau\,$, is part of \cite{Che}, Theorem~4.3. 
(In \cite{Che}, $\, \theta^{-1} \,$ and $\,\tau \,$ are denoted by $\,\tau_+\,$ and $\,\tau_{-}\,$, 
respectively). The following calculation now completes the proof:
\begin{eqnarray*}
(1,1;g_1)(\alpha,\beta;\text{id})(X_i,Y_i,T_i)&=& (\alpha\beta Y_i X_i, \beta Y_i, T_i)\\
                                              &=&(\alpha\beta,\beta;\text{id})(1,1;g_1)(X_i,Y_i,T_i)\\
                                              &=& (g_1(\alpha,\beta);g_1)(X_i,Y_i,T_i) \\*[1ex]
(1,1;g_2)(\alpha,\beta;\text{id})(X_i,Y_i,T_i)&=&(\alpha X_i, \alpha^{-1}\beta X_i^{-1}Y_i, T_i)\\
                                              &=&(\alpha,\alpha^{-1}\beta;\text{id})(1,1;g_2)(X_i,Y_i,T_i)\\
                                              &=&(g_2(\alpha,\beta);g_2)(X_i,Y_i,T_i)\ . \\
\end{eqnarray*}
\end{proof}
Note that any automorphism of $ H $ in the image of $ \Phi_{H} $ fixes the generators 
$\,T_1,...,T_{n-1}\,$. It follows that the $G$-action on $H$ defined by $ \Psi_H $ induces a 
$G$-action on the spherical algebra $\rrdaha$ and the left $\rrdaha$-module $\, \varepsilon P 
= \rldaha \,\oplus\, \rrdaha \,$. In other words, we have group homomorphisms 
$\,\Phi_{\rrdaha}:\, G \rightarrow \Aut_{\c}(\rrdaha)\,$ and 
$\,\Phi_{P}:\,G \rightarrow \Aut_{\c}(\varepsilon P)\,$. In addition, we have
\begin{lemma} 
\la{gpact2}
The following assignment extends to a well-defined group homomorphism $\, 
\Phi_{\Lambda}:\,G \to \Aut_\c(\Lambda)\,$:
\begin{eqnarray*}
(1,1;g_1) &\mapsto&  (X \mapsto XY \,\,,\,\,Y \mapsto Y\,\,,\,\,j \mapsto Y^{-1}j\,\,,\,\,i \mapsto i\,Y)\ , \\
(1,1;g_2) &\mapsto&  (X \mapsto X \,\,,\,\,Y \mapsto YX^{-1} \,\,,\,\,j \mapsto X\,j\,\,,\,\,i \mapsto i\,X^{-1})\ , \\
 (\alpha,\beta;\text{id}) &\mapsto& (X \mapsto \alpha X \,\,,\,\,Y
 \mapsto \beta Y\,\,,\,\,j \mapsto j \,\,,\,\,i \mapsto i)\ .
\end{eqnarray*}
\end{lemma}
\begin{proof} 
A direct calculation similar to that of Lemma~\ref{gpact1}. We leave details to the reader.
\end{proof}

Note that the above action of $ G $ on $ \Lambda $ preserves the idempotents $ e $ and $ \ei $
and hence restricts to the subalgebra $\, e \Lambda e \subset \Lambda \,$. 
By Proposition~\ref{pbimod}, $\,\varepsilon P \,$ is an $\varepsilon H \varepsilon $-$\Lambda$-bimodule. 
The subspace $\,\varepsilon P e := \varepsilon\,P \otimes_\Lambda \Lambda e \subseteq \varepsilon P\,$
is preserved by any automorphism in $ \Aut_{\c}(\varepsilon P) $, which is in the image of $\Phi_{P} $. 
As a result, we have an action of $G$ on $\,\varepsilon P e\,$. Now, form the semidirect product $\,[\rrdaha \otimes (e
\Lambda e)^{\circ}] \rtimes G\,$, with $G$ acting diagonally on $\,\rrdaha\,$ and $\,e\Lambda e\,$ as 
in Lemma~\ref{gpact1} and Lemma~\ref{gpact2}.
\begin{proposition} 
\la{gpact3}
The action of $ G $ on $\,\varepsilon P e\,$ defined above makes it a $\,[\rrdaha \otimes (e
\Lambda e)^{\circ}] \rtimes G$-module (equivalently, $\,\varepsilon P e\,$ is a
$G$-equivariant $\rrdaha$-$e\Lambda e$-bimodule).
\end{proposition}
\begin{proof}
We need to verify that, for all $\,h \in \rrdaha $, $\,m \in \varepsilon P e$, $\,x 
\in e \Lambda e $ and $ \sigma \in G$,
\begin{equation} 
\la{equivar1}
(\sigma.h)(\sigma.m)(\Theta(\sigma.x))= \sigma.(h m \Theta(x))
\text{. }
\end{equation} 
Note that $\,\varepsilon P e\,$ may be identified with the direct summand $\rldaha$ 
of $\varepsilon P$. The $G$-action on $\rldaha$ as well as that on
$\rrdaha$ are obtained by restricting $\Phi_{H}$. It follows that $\,\sigma(h.m)=
(\sigma.h)(\sigma.m)\,$ for all $\,\sigma \in G$,$\,h \in
\rrdaha\,$ and $\,m \in \varepsilon P e\,$. Therefore, it suffices 
to verify \eqref{equivar1} with $\,h=1\,$. For this, it suffices to show that 
$\,\Theta: (e \Lambda e)^{\circ} \to \lldaha\,$ is a $G$-module homomorphism. 
This boils down to a trivial calculation which we leave to the reader.
\end{proof}

\begin{proof}[Proof of Proposition~\ref{Teqv}]
To simplify the notation, we set 
$\,\Gamma := \rrdaha \otimes (e\Lambda e)^{\circ}\,$ and 
$\, \Gamma_A := \rrdaha \otimes A_q^{\circ} \,$. Note that $G$ acts 
naturally on $A_q$. Further, the algebra homomorphism 
$\,e \Lambda e \rightarrow A_q\,$ given by Lemma~\ref{laq} is $G$-equivariant.
Hence, the induced homomorphism $\,\Gamma \to \Gamma_A\,$ is
$G$-equivariant. Now, for any $G$-equivariant $ \Gamma$-module $M$,
we have an isomorphism $ \Gamma_A$-modules
$$
M \otimes_{e \Lambda e} A_q \cong 
\Gamma_A \otimes_{\Gamma} M \cong
(\Gamma_A \rtimes G) \otimes_{\Gamma \rtimes G} M\ .
$$ 
It follows that $\, M \Lotimes_{e \Lambda e} A_q \cong (\Gamma_A \rtimes G) 
\Lotimes_{\Gamma \rtimes G} M\,$ in the derived category of $\rrdaha$-$A_q$-bimodules.
Hence $\,M \Lotimes_{e\Lambda e} A_q\,$ is a $G$-equivariant $\rrdaha$-$A_q$-bimodule. 
By Proposition~\ref{gpact3}, the $\rrdaha$-$A_q$-bimodule $\,\tilde{P}:= 
\varepsilon P e \Lotimes A_q\,$ is $G$-equivariant. Identify 
$\,\rrdaha \cong \O(\CC_n^q) \,$ as in Proposition~\ref{oblom}, and let $\,\mu: 
\O \otimes \tilde{P} \otimes A_q \to \tilde{P}\,$ denote the structure map
of the bimodule $ \tilde{P} $ with this identification. Now, 
for any $\,\sigma \in G\,$ and for any maximal ideal $\,\m \subset \O\,$, 
the commutatuve diagram
$$
\begin{diagram}[small, tight]
\O \otimes \tilde{P} \otimes A_q &  \rTo^{\mu} & \tilde{P} \\
\dTo^{\sigma\otimes \sigma \otimes \sigma}     &             & \dTo_{\sigma}  \\
\O \otimes \tilde{P} \otimes A_q &  \rTo^{\mu} & \tilde{P} 
\end{diagram}
$$
induces 
$$
\begin{diagram}[small, tight]
(\O/\m) \otimes [\,(\O/\m) \otimes_{\O} \tilde{P}\,] \otimes A_q &  \rTo^{\mu} & (\O/\m) \otimes_\O \tilde{P} \\
\dTo^{\sigma\otimes \sigma \otimes \sigma}     &             & \dTo_{\sigma}  \\
(\O/\m)_\sigma \otimes [\,(\O/\m)_\sigma \otimes_{\O} \tilde{P}\,] \otimes A_q &  \rTo^{\mu} & (\O/\m)_\sigma \otimes_\O \tilde{P} 
\end{diagram}
$$
where $\,(\O/\m)_\sigma \,$ denotes the twisting by $ \sigma $ of the $ \O$-module $\, \O/\m \cong \c \,$.
It follows that the map $\,\omega_n:\CC^q_n \rightarrow \RR(A_q)\,$ is $G$-equivariant. 
Proposition~\ref{Teqv} follows once we note that, by Theorem~\ref{Pic}, $\,\Pic(A_q) \cong G/\Z^2 \,$, 
and the $G$-action on $\CC^q_n$ descends to the $ \Pic(A_q)$-action on $\bCC^q_n$.
\end{proof}
\section{Appendix: Comparison with Noncommutative Tori in the Smooth Case}
\la{S6}
In the following table, we compare the properties of algebraic and smooth
noncommutative tori. In the algebraic case, most ring-theoretic and
homological properties follow from the fact that $A_q$ is a simple hereditary
domain (see \cite{J}); the classification of projectives and Morita
classification are results of this paper, and the computation of
Hochschild and cyclic homology can be found in \cite{Wa}. In the
smooth case, the description of projective modules and Morita classification
can be found in \cite{R2}, the Picard group of $ \AA_q $ is computed in \cite{K},
and results on Hochschild, cyclic homology, and cohomological dimension follow
from \cite{C}.
\begin{table}[ht]
\caption{Algebraic vs Smooth Noncommutative Tori} \centering
\setlength{\extrarowheight}{4pt}
\begin{tabularx}{\linewidth}%
{|c|>{\setlength\hsize{1\hsize}}X%
|>{\setlength\hsize{1\hsize}}X|} \hline Property & Algebraic torus
$A_q$ & Smooth torus ${\mathcal {A}}_\alpha$\newline ($q=e^{2\pi
i\alpha}$)
\\
\hline
Cancellation property& Any f.g. projective module of Rank ($\geq 2$) is free. & Stably isomorphic f.g. projective
modules are isomorphic.\\
\hline
f.g. projective modules& Any f. g. projective module is either free or isomorphic to a right ideal. &Any f.g.
projective module is isomorphic to a standard module constructed by Connes. \\
\hline $K_0$ &$\mathbb{Z}$
&$\mathbb{Z}^2$\\
\hline Isomorphism class&$A_q$ is isomorphic to $A_{q'}$ if and only
if $q=q'^{\pm 1}$.&${\mathcal {A}}_\alpha$ is isomorphic to
${\mathcal {A}}_{\alpha'}$ if and only if $\alpha\pm \alpha'\in
\mathbb{Z}$.
\\
\hline Morita class& $A_q$ is Morita equivalent to $A_{q'}$ if and
only if $q=q'^{\pm1}$. A unital algebra is Morita equivalent to
$A_q$ if and only if $A$ is isomorphic to $M_n(A_q)$ for some $n\in
\mathbb{N}$ or $\End_{A_q}(R)$ for some right ideal $R$ of $A_q$. &
A unital $C^*$-algebra is Morita equivalent to ${\mathcal
{A}}_\alpha$ if and only if it is isomorphic $M_n({\mathcal
{A}}_{\alpha'})$ for some $n\in \mathbb{N}$, and
$\alpha=\frac{a\alpha'+b}{c\alpha'+d}$, where
$\left(\begin{array}{ll}a&b\\ c&d\end{array}\right)\in \GL(2,
\mathbb{Z})$.
\\
\hline Outer automorphism& $({\mathbb {C}^*}/{\mathbb{Z}})^2\rtimes
\SL(2, \mathbb{Z})$& $({\mathbb {C}^*}/{\mathbb{Z}})^2\rtimes \SL(2,
\mathbb{Z})$
\\
\hline Picard group&$\operatorname{Out}(A_q)$&  If $\alpha$ is not
quadratic, $\Pic({\mathcal {A}}_\alpha)=\operatorname{Out}({\mathcal
{A}}_\alpha)$; if $\alpha$ is quadratic, $\Pic({\mathcal
{A}}_\alpha)=(\operatorname{Out}({\mathcal {A}}_\alpha))\rtimes
\mathbb{Z}.$
\\
\hline Hochschild homology&$HH_n(A_q)$\newline
=$\left\{\begin{array}{ll} \mathbb{C}& n=0,2\\ \mathbb{C}^2&n=1\\
0&\text{otherwise}.\end{array}\right.$& $HH_n({\mathcal
{A}}_\alpha)$ (when $\alpha$ satisfies a diophantine
condition)\newline =$\left\{\begin{array}{ll} \mathbb{C}& n=0,2\\
\mathbb{C}^2&n=1\\ 0&\text{otherwise}.\end{array}\right.$\newline
When $\alpha$ does not satisfy a diophantine condition,
$HH_0({\mathcal{A}}_\alpha)$ and $HH_1({\mathcal {A}}_\alpha)$ are
infinitely dimensional.
\\
\hline Cyclic homology& $HP_{\text{odd}}(A_q)={\mathbb
{C}}^2$\newline $HP_{\text{even}}(A_q)={\mathbb {C}}^2$ &
$HP_{\text{odd}}({\mathcal {A}}_\alpha)={\mathbb {C}}^2$\newline
$HP_{\text{even}}({\mathcal {A}}_\alpha)={\mathbb {C}}^2$
\\
\hline Cohomological dimension & 1 &2
\\
\hline
\end{tabularx}
\end{table}


\begin{thebibliography}{45}
%
\bibitem[AD1]{AD1}
J. Alev and F. Dumas, \textit{Automorphismes de certains compl\'et\'es du
corps de Weyl quantique}, Collect. Math. \textbf{46}(1-2) (1995),  1--9.
%
\bibitem[AD2]{AD2}
J. Alev and F. Dumas, \textit{Rigidit\'e des prolongements des quotients
primitifs minimaux de $ U_q(\mathfrak{sl}(2)) $ dans l'alg\`ebre quantique de Weyl-Hayshi},
Nagoya Math. J. \textbf{143} (1996), 119--146.
%
\bibitem[BaEG]{BaEG}
V.~Baranovsky, S.~Evens, and V.~Ginzburg,
\textit{Representations of quantum tori and $G$-bundles on elliptic curves}
in \textit{The orbit method in geometry and physics}, Progr. Math. \textbf{213},
Birkh\"auser, Boston, 2003, pp. 29--48.
%
\bibitem[BC]{BC}
Yu.~Berest and O.~Chalykh, \textit{${\sf A}_{\infty}$-modules and
Calogero-Moser spaces},  J. reine angew. Math. \textbf{607} (2007), 69--112.
%
\bibitem[BC1]{BC1}
Yu.~Berest and O.~Chalykh, \textit{Ideals of quantum Weyl algebras},  
unpublished notes, 2004.
%
\bibitem[BCE]{BCE}
Yu. Berest, O. Chalykh and F. Eshmatov, \textit{Recollement of the
deformed preprojective algebras and the Calogero-Moser
correspondence}, Moscow Math. J. \textbf{8}(1) (2008), 21--37.
%
\bibitem[BEG]{BEG}
Yu. Berest, P. Etingof and V. Ginzburg, \textit{Morita equivalence of
Cherednik algebras}, J. reine angew. Math. \textbf{568} (2004), 81--98.
%
\bibitem[BW1]{BW1}
Yu. Berest and G. Wilson, \textit{Automorphisms and ideals of the
Weyl algebra}, Math. Ann. \textbf{318}(1) (2000), 127--147.
%
\bibitem[BW2]{BW2}
Yu. Berest and G. Wilson, \textit{Ideal classes of the Weyl
algebra and noncommutative projective geometry} (with an Appendix
by M. van den Bergh), Internat. Math. Res. Notices \textbf{26}
(2002), 1347--1396.
%
%
\bibitem[CBS]{CBS}
W. Crawley-Boevey and P. Shaw, \textit{Multipicative preprojective
algebras, middle convolution and the Deligne-Simpson problem}, Adv.
Math. \textbf{201}(1) (2006), 180--208.
%
\bibitem[CH]{CH} R. C. Cannings and M. P. Holland,
\textit{Etale covers, bimodules and differential operators},
Math.~Z.~\textbf{216} (1994), 179--194.
%
\bibitem[CH1]{CH1} R. C. Cannings and M. P. Holland,
\textit{Rings of operators on modules over commutative rings and
their right ideals}, J. Algebra \textbf{186} (1996), 235--263.
%
%
%
\bibitem[CN]{CN} O.~Chalykh and F.~Nijhoff, \textit{Bispectral rings of
difference operators}, Russian Math. Surveys \textbf{54}(3) (1999), 644--645. 
%
\bibitem[C1]{C1}
I. Cherednik, \textit{Double affine Hecke algebras, Knizhnik-Zamolodchikov equations, 
and Macdonald operators}, IMRN \textbf{9} (1992), 171--180.
%
\bibitem[C2]{Che}
I. Cherednik, \textit{Macdonald's evaluation
conjectures and difference Fourier transform}, Invent. Math.
\textbf{122} (1995), 119-145.
%
\bibitem[Co]{C}
A.~Connes, \textit{Noncommutative Geometry}, Academic Press Inc.,
San Diego, CA, 1994.
%
\bibitem[D]{D}
J. Dixmier, \textit{Sur les alg\`ebres de Weyl},
Bull. Soc. Math. France \textbf{96} (1968), 209--242.
%
%
\bibitem[EG]{EG}
P.~Etingof and V.~Ginzburg, \textit{Symplectic reflection
algebras, Calogero-Moser space, and deformed Harish-Chandra
homomorphism}, Invent. Math. \textbf{147} (2002), 243--348.
%
\bibitem[EGGO]{EGGO}
P.~Etingof, W.~Gan, V.~Ginzburg and A.~Oblomkov,
\textit{Harish-Chandra homomorphisms and symplectic reflection algebras
for wreath-products}, Publ. Math. Inst. Hautes Etudes Sci. \textbf{105} (2007), 
91--155.
%
\bibitem[F]{F}
A.~Fr\"ohlich, \textit{The Picard group of noncommitative rings,
in particular of orders}, Trans. Amer. Math. Soc. \textbf{180}
(1973), 1--45.
%
\bibitem[FR]{FR}
V. V. Fock and A. A. Rosly, \textit{Poisson structure on moduli of flat 
connections on a Riemann surfaces and the $r$-matrix}, Moscow
Seminar in Mathematical Physics, Amer. Math. Soc. Transl. Ser. 2, 
\textbf{191}, Amer. Math. Soc., Providence RI, 1999, pp. 67--86.
%
\bibitem[FGL]{FGL}
C.~Frohman, R.~Gelca and W.~Lofaro, \textit{The $A$-polynomial
from the noncommutative viewpoint}, Trans. Amer. Math. Soc.
\textbf{354}(2) (2001), 735--747.
%
\bibitem[G]{G} S.~Gukov, \textit{Three-dimensional quantum gravity,
Chern-Simons theory, and the $A$-polynomial},  Comm. Math. Phys.
\textbf{255}(3)  (2005), 577--627.
%
\bibitem[J]{J} V.~Jategaonkar,  \textit{A multiplicative
analog of the Weyl algebra}, Comm. Algebra \textbf{12}(14) (1984),
1669--1668.
%
\bibitem[KKO]{KKO}
A.~Kapustin, A.Kuznetsov, and D.~Orlov, \textit{Noncommutative
instantons and twistor transform}, Comm. Math. Phys. \textbf{220}
(2001), 385--432.
%
\bibitem[K1]{K} K.~Kodaka, \textit{Picard groups of irrational rotation
$C^*$-algebras}, J. London Math. Soc. (2) \textbf{56} (1997), 179--188.
%
\bibitem[K2]{K2} K.~Kodaka, \textit{$C^*$-algebras that are isomorphic after
tensoring and full projections}, Proc. Edinb. Math. Soc. Math. Soc. (2)
\textbf{47}(3) (2004), 659--668.
%
\bibitem[NB]{NvdB}
K. de Naeghel and M. Van den Bergh, \textit{Ideals classes of
three-dimensional Sklyanin algebras}, J.~Algebra \textbf{276} (2)
(2004), 515--551.
%
\bibitem[NS]{NS}
T.~A.~Nevins and J.~T.~Stafford, \textit{Sklyanin algebras and
Hilbert schemes of points}, Adv. Math. \textbf{210}(2) (2007),
405--478.
%
\bibitem[O]{O} A.~Oblomkov, \textit{Double affine Hecke algebras and
Calogero-Moser spaces},
Represent. Theory \textbf{8} (2004), 243--266.
%
\bibitem[Or]{Or}
D.~O.~Orlov, \textit{Derived categories of coherent sheaves on abelian varieties and
equivalences between them}, Izv. Ross. Akad. Nauk Ser. Mat. \textbf{66} (2002), 131--158.
%
\bibitem[Po]{Po} A. Polishchuk,
\textit{Noncommutative two-tori with real multiplication as noncommutative projective varieties},
J. Geom. Phys. \textbf{50} (2004), 162--187.
%
%
\bibitem[RS]{RS}
L. Richard and A. Solotar, \textit{Isomorphisms between quantum generalized Weyl algebras},
J. Algebra Appl. \textbf{5}(3) (2006), 271--285.
%
\bibitem[R1]{R1}
M.~Rieffel, \textit{$C^*$-algebras associated with irrational rotations},
Pacific J. Math. \textbf{93}(2) (1981), 415--429.
%
\bibitem[R2]{R2}
M.~Rieffel, \textit{The cancellation theorem for projective
modules over irrational rotation algebras}, Proc. London Math.
Soc. (3) \textbf{47} (1983), 285--302.
%
\bibitem[R3]{R3}
M.~Rieffel, \textit{Noncommutative tori---a case study of noncommutative differentiable manifolds}
in \textit{Geometric and topological invariants of elliptic operators}, Contemp. Math. \textbf{105},
Amer. Math. Soc., Providence, RI, 1990, pp. 191--211.
%
\bibitem[S]{S}
P.~Samuelson, \textit{Noncommutative $A$-polynomials and quantum Calogero-Moser spaces}, in preparation.
%
\bibitem[ST]{ST}
P. Seidel and R. Thomas, \textit{Braid group actions on derived
categories of coherent sheaves}, Duke Math. J. \textbf{108} (2001),
no. 1, 37-108.
%
\bibitem[SV]{SV}
Y. Soibelman and V. Vologodsky, \textit{Noncommutative compactifications and
elliptic curves}, IMRN \textbf{28} (2003), 1549--1569.
%
\bibitem[St]{Sta} J.~T.~Stafford, \textit{Endomorphisms of right ideals
of the Weyl algebra}, Trans. Amer. Math. Soc. \textbf{299}(2)
(1987), 623--639.
%
\bibitem[vdB]{vdB}
M. Van den Bergh,
\emph{Noncommutative quasi-Hamiltonian spaces}, Comtemp. Math. \textbf{450}
(2008) 273--299.
%
\bibitem[Wa]{Wa}
M. Wambst, \textit{Hochschild and cyclic homology of the quantum 
multiparametric torus},
J. Pure Appl. Algebra \textbf{114} (1997), no. 3, 321--329.
%
\bibitem[We]{We}
D.~B.~Webber, \textit{Ideals and modules of simple Noetherian
hereditary rings}, J. Algebra \textbf{16} (1970), 239--242.
%
\bibitem[W1]{W}
G. Wilson, \textit{Collisions of Calogero-Moser particles and an
adelic Grassmannian} (with an Appendix by I. G. Macdonald),
Invent. Math. \textbf{133} (1998), 1--41.
%
\bibitem[W2]{W2}
G. Wilson,
\textit{Bispectral commutative ordinary differential
operators}, J. Reine Angew. Math. \textbf{442} (1993), 
177--204.
%
%
\end{thebibliography}
\end{document}